\documentclass[12pt]{amsart}
\usepackage{amssymb}
\usepackage{amsfonts}

\newtheorem{theorem}{Theorem}
\theoremstyle{plain}

\newtheorem{corollary}{Corollary}

\newtheorem{definition}{Definition}

\newtheorem{lemma}{Lemma}

\newtheorem{proposition}{Proposition}

\numberwithin{equation}{section}
\input{tcilatex}
\newcommand{\stft}{short-time Fourier transform}

\newcommand{\fif}{if and only if}
\newcommand{\tfs}{time-frequency shift}

\newcommand{\modsp}{modulation space}

\newcommand{\beqa}{\begin{eqnarray*}}
\newcommand{\eeqa}{\end{eqnarray*}}

\newcommand{\field}[1]{\mathbb{#1}}
\newcommand{\bR}{\field{R}}        %  real numbers
\newcommand{\bZ}{\field{Z}}        %  whole numbers 
\newcommand{\bC}{\field{C}}        %  complex numbers
 \def\cF{\mathcal{F}}              % Calligraphic Letters

 \def\cB{\mathcal{B}}

 \def\cL{\mathcal{L}}

\def\<{\left<}
\def\>{\right>}

\def\inv{^{-1}}

\def\mv1{M_v^1}

\begin{document}
\title[On Gabor frames with Hermite functions]{ Banach Gabor frames with
Hermite functions: polyanalytic spaces from\ the Heisenberg group}
\author{Lu\'{\i}s Daniel Abreu}
\address{CMUC, Department of Mathematics of University of Coimbra, School of
Science and Technology (FCTUC) 3001-454 Coimbra, Portugal }
\email{daniel@mat.uc.pt}
\urladdr{http://www.mat.uc.pt/\symbol{126}daniel/}
\author{Karlheinz Gr\"{o}chenig}
\address{Faculty of Mathematics, University of Vienna, Nordbergstrasse 15,
A-1090 Wien, Austria}
\email{karlheinz.groechenig@univie.ac.at}
\urladdr{http://homepage.univie.ac.at/karlheinz.groechenig/}
\thanks{L.\ D.\ A.\ research was partially supported by CMUC/FCT and FCT
project PTDC/MAT/114394/2009, POCI 2010 and FSE. K.\ G.\ was supported in
part by the National Research Network S106 SISE of the Austrian Science
Foundation (FWF)}
\date{}
\dedicatory{Dedicated to Paul L. Butzer on the occasion of his 80th birthday}

\begin{abstract}
Gabor frames with Hermite functions are equivalent to sampling sequences in
true Fock spaces of polyanalytic functions. In the $L^{2}$-case, such an
equivalence follows from the unitarity of the polyanalytic Bargmann
transform. We will introduce Banach spaces of polyanalytic functions and
investigate the mapping properties of the polyanalytic Bargmann transform on
modulation spaces. By applying the theory of coorbit spaces and localized
frames to the Fock representation of the Heisenberg group, we derive
explicit polyanalytic sampling theorems which can be seen as a polyanalytic
version of the lattice sampling theorem discussed by J. M. Whittaker in
Chapter 5 of his book \emph{Interpolatory Function Theory}.
\end{abstract}

\maketitle

\section{Introduction}

In this note we will be concerned with \emph{Gabor expansions} of the form%
\begin{equation}
f(t)=\sum_{l,k\in 
%TCIMACRO{\U{2124} }%
%BeginExpansion
\mathbb{Z}
%EndExpansion
}c_{k,l}e^{2\pi i\alpha lt}h_{n}(t-\beta k)\text{,}
\label{HermiteExpansions}
\end{equation}%
where $\alpha $ and $\beta $ are real constants and $h_{n}$ are the \emph{%
Hermite functions}%
\begin{equation*}
h_{n}(t)=c_{n}e^{\pi t^{2}}\left( \frac{d}{dt}\right) ^{n}\left( e^{-2\pi
t^{2}}\right)
\end{equation*}
where $c_{n}$ is chosen so that $\|h_n \|_2 =1$. Expansions of type (\ref%
{HermiteExpansions}) are useful for the multiplexing of signals \cite{Balan}%
, \cite{CharlyYurasuper}, \cite{Abreusampling} and image processing \cite{G}.

Gabor expansions with Hermite functions 
% enjoy rich "soft" (functional/group
% theoretical) and "hard" analytic structures (complex variables). They
have been introduced in mathematical time-frequency analysis \cite%
{CharlyYura} and studied further in \cite{Fuhr}, \cite{CharlyYurasuper} and 
\cite{Abreusampling}, with an emphasis on vector-valued Gabor frames
(so-called Gabor super-frames).

The properties of Gabor expansions with Hermite functions are the\ result
of\ an interplay between classical analysis (orthogonal polynomials,
Weierstrass sigma functions) and modern mathematical methods (frame theory,
group representations, and modulation spaces).

It has been discovered recently \cite{Abreusampling} that the construction
of expansions of type (\ref{HermiteExpansions}) is equivalent to sampling
problems in spaces of functions which satisfy the generalized Cauchy-Riemann
equations of the form%
\begin{equation}
\left( \frac{d}{d\overline{z}}\right) ^{n}F(z)=\frac{1}{2^{n}}\left( \frac{%
\partial }{\partial x}+i\frac{\partial }{\partial \xi }\right) ^{n}F(x+i\xi
)=0\text{.}  \label{eq:c1}
\end{equation}%
This is equivalent to saying that $F(z)$\ is a polynomial of order $n-1$ in $%
\overline{z}$\ with analytic functions $\{\varphi _{k}(z)\}_{k=0}^{n-1}$\ as
coefficients:%
\begin{equation}
F(z)=\sum_{k=0}^{n-1}\overline{z}^{k}\varphi _{k}(z)  \label{polypolynomial}
\end{equation}%
Such functions are known as \emph{polyanalytic functions. }They have been
investigated thoroughly, notably by the Russian school led by Balk \cite%
{Balk}, and they provide extensions of classical operators from complex
analysis \cite{BegehrHille}. The connection of polyanalytic function theory
to Gabor expansions seems to be yet another instance of how --- in the words
of Folland \cite{Folland} --- "the abstruse meets the applicable" in
time-frequency analysis. Indeed, time-frequency analysis is prone to reveal
unexpected relations to other fields of mathematics. Two recent examples are
the associations with\ Banach algebras \cite{GL} and with noncommutative
geometry \cite{Luef}.

Our objective is to investigate the connection between Gabor frames and
polyanalytic function theory in more detail. Our main contributions are the
following:%  we now depart from the Hilbert
% space theory of \cite{Abreusampling} and enter Banach space theory. More
% precisely we will:

\begin{enumerate}
\item We will obtain explicit formulas for the sampling and interpolation of
polyanalytic functions on a lattice in the complex plane.

\item We extend the theory of the polyanalytic Fock space from Hilbert space
to Banach spaces and study polyanalytic functions satisfying $L^{p}$%
-condition. Precisely, we investigate the space $\mathcal{F}_{p}^{n}(%
%TCIMACRO{\U{2102} }%
%BeginExpansion
\mathbb{C}
%EndExpansion
)$\ consisting of all polyanalytic functions $F$ of order $n$, i.e.,
satisfying~\eqref{eq:c1}, such that 
\begin{equation*}
\left\Vert F\right\Vert ^{p}=\int_{%
%TCIMACRO{\U{2102} }%
%BeginExpansion
\mathbb{C}
%EndExpansion
}\left\vert F(z)\right\vert ^{p}e^{-\pi p\frac{\left\vert z\right\vert }{2}%
^{2}}dz
\end{equation*}%
is finite. For this purpose we use the theory of modulation spaces \cite%
{FeiModulation} and develop the $L^{p}$-theory of the polyanalytic Bargmann
transform which so far has only been studied in the $L^{2}$ case \cite%
{Abreusampling}.

\item We then extend the frame and sampling expansions beyond the Hilbert
space. The tools are taken from coorbit theory \cite{FG0}, \cite{FG}, \cite%
{atomGroch} and the theory of localized frames \cite{Localization_one}.
\end{enumerate}

As a specific result we state a polyanalytic version of a sampling formula
for a lattice. The sampling theorem is in the spirit of the fundamental
Whittaker-Shannon-Kotel\'{}nikov formula, as discussed for instance in
Whittaker's classical treatise \emph{Interpolatory Function Theory }\cite{W}%
. Let $\Lambda =\alpha (\bZ+i\bZ)$ be the square lattice consisting of the
points $\lambda =\alpha l+i\alpha m$, $k,m\in \bZ$ and let $\sigma _{\Lambda
^{\circ }}(z)$ be the classical Weierstrass sigma function associated to the
adjoint lattice $\Lambda ^{\circ }=\alpha \inv(\bZ+i\bZ)$ and set 
\begin{equation*}
S_{\Lambda ^{0}}^{n}(z)=\left( \frac{\pi ^{n}}{n!}\right) ^{\frac{1}{2}%
}e^{\pi \left\vert z\right\vert ^{2}}\left( \frac{d}{dz}\right) ^{n}\left[
e^{-\pi \left\vert z\right\vert ^{2}}\frac{\left( \sigma _{\Lambda
^{0}}(z)\right) ^{n+1}}{n!z}\right] .
\end{equation*}%
%
%
%
%
%
% \emph{and }$\sigma _{\Lambda ^0}(z)$\emph{\ the Weierstrass sigma function
% associated with the dual lattice }$\Lambda ^0=\{\frac{l}{\alpha }+i\frac{m%
% }{\beta }\}$\emph{, }

\textbf{Theorem.}\emph{\ If }$\alpha ^2 <\frac{1}{n+1}$\emph{, then every }$%
F\in \mathcal{F}_{p}^{n+1}(%
%TCIMACRO{\U{2102} }%
%BeginExpansion
\mathbb{C}
%EndExpansion
)$\emph{\ possesses the sampling expansion}%
\begin{equation*}
F(z)=\sum_{\lambda \in \alpha (\bZ +i\bZ )} F(\lambda )e^{\pi \overline{%
\lambda }z-\pi \left\vert \lambda \right\vert ^{2}/2} \, S_{\Lambda ^0
}^{n}(z-\lambda ),
\end{equation*}%
\emph{with unconditional convergence in $\cF ^n_p(\bC )$ for $1\leq p<\infty 
$. }

\vspace{ 3 mm}

The outline of the paper is as follows. In section 2, we explain the
connection between the \stft with Hermite functions and the true
polyanalytic Bargmann transform. We extend the $L^{2}$-theory of the true
polyanalytic Bargmann transform to a general class of Banach spaces of
polyanalytic functions in section 3. Then section 4 studies Gabor frames
with Hermite functions in $L^{2}(%
%TCIMACRO{\U{211d} }%
%BeginExpansion
\mathbb{R}
%EndExpansion
)$. We provide a different proof of the sufficient condition given in \cite%
{CharlyYura} for the lattice parameters which generate those frames. In the
last section we study Gabor Banach frames with Hermite functions using
coorbit theory and localized frames associated with the representations of
the Heisenberg group, and derive the corresponding sampling theorems for the
polyanalytic Fock spaces.

\section{Gabor transforms with Hermite functions}

\subsection{The Bargmann transform}

Expansions of the type (\ref{HermiteExpansions}) are closely related to the
samples of the \stft\ of $f$ with respect to the Hermite windows\ $h_{n}$.

We recall that the \stft\ of a function or distribution $f$ with respect to
a window function $g$ is defined to be 
\begin{equation}
V_{g}f(x,\xi )=\int_{%
%TCIMACRO{\U{211d} }%
%BeginExpansion
\mathbb{R}
%EndExpansion
}f(t)\overline{g(t-x)}e^{-2\pi i\xi t}dt.  \label{Gabor}
\end{equation}%
If we choose the Gaussian function $h_{0}(t)=2^{\frac{1}{4}}e^{-\pi t^{2}}$
as a window in \eqref{Gabor}, then a simple calculation shows that 
\begin{equation}
e^{-i\pi x\xi +\pi \frac{\left\vert z\right\vert ^{2}}{2}}V_{h_{0}}f(x,-\xi
)=\int_{%
%TCIMACRO{\U{211d} }%
%BeginExpansion
\mathbb{R}
%EndExpansion
}f(t)e^{2\pi tz-\pi z^{2}-\frac{\pi }{2}t^{2}}dt=\mathcal{B}f(z)\text{.}
\label{Bargmann}
\end{equation}%
Here $\mathcal{B}f (z)$ is the usual Bargmann transform of $f$, see for
instance \cite{folland89}. The Bargmann transform $\mathcal{B}f$\ is an
entire function and thus satisfies the Cauchy-Riemann equation 
\begin{equation*}
\frac{d}{d\overline{z}}\, \mathcal{B}f=0\text{.}
\end{equation*}%
Furthermore, $\mathcal{B}$ is an unitary isomorphism from $L^{2}(\bR )$ onto
the Bargmann-Fock space $\cF (\bC )$ consisting of all entire functions
satisfying 
\begin{equation}
\left\Vert F\right\Vert _{\mathcal{F}(%
%TCIMACRO{\U{2102} }%
%BeginExpansion
\mathbb{C}
%EndExpansion
)}^{2}=\int_{%
%TCIMACRO{\U{2102} }%
%BeginExpansion
\mathbb{C}
%EndExpansion
}\left\vert F(z)\right\vert ^{2}e^{-\pi \left\vert z\right\vert
^{2}}dz<\infty .  \label{Focknorm}
\end{equation}

\subsection{The polyanalytic Bargmann transform}

Now choose a general Hermite function \ $h_{n}$ as a window for the \stft\
in \eqref{Gabor}. A calculation (see \cite{BS93} or \cite{CharlyYura} for
details) shows that%
\begin{equation}
e^{-i\pi x\xi +\frac{\pi }{2}\left\vert z\right\vert ^{2}}V_{h_{n}}f(x,-\xi
)=\left( \frac{\pi ^{n}}{n!}\right) ^{\frac{1}{2}}\sum_{0\leq k\leq n}\binom{%
n}{k}(-\pi \overline{z})^{k}\left( \frac{d}{dz}\right) ^{n-k}(\mathcal{B}%
f)(z)=F(z)\text{.}  \label{gaborhermite1}
\end{equation}%
Now we have a serious obstruction for exploiting complex analysis tools. The
function $F$ on the right hand side of (\ref{gaborhermite1}) is no longer
analytic. However, after differentiating (\ref{gaborhermite1}) $n+1$ times
with respect to $\overline{z}$, we see that $F$ satisfies the equation%
\begin{equation}
\left( \frac{d}{d\overline{z}}\right) ^{n+1}F(z)=0\text{.}
\label{polyanalytic}
\end{equation}%
A function $F$ satisfying (\ref{polyanalytic}) is called polyanalytic of
order $n+1$. By using the Leibnitz formula, we may write (\ref{gaborhermite1}%
) more compactly as 
\begin{equation}
e^{-i\pi x\xi +\frac{\pi }{2}\left\vert z\right\vert ^{2}}V_{h_{n}}f(x,-\xi
)=\left( \frac{\pi ^{n}}{n!}\right) ^{\frac{1}{2}}e^{\pi \left\vert
z\right\vert ^{2}}\frac{d^{n}}{dz^{n}}\left[ e^{-\pi \left\vert z\right\vert
^{2}}\mathcal{B}f(z)\right] \text{.}  \label{eq:hmhm}
\end{equation}%
This motivates the following definition:%
\begin{equation}
\mathcal{B}^{n+1}f(z)=e^{-i\pi x\xi +\frac{\pi }{2}\left\vert z\right\vert
^{2}}V_{h_{n}}f(x,-\xi ).  \label{polyBargmann}
\end{equation}%
The map $\mathcal{B}^{n}$ is called the \emph{true polyanalytic Bargmann
transform} of order $n$ and has been studied in \cite{Abreusampling} and 
\cite{Abreustructure}. As a consequence of the orthogonality relations for
the \stft\ it was shown that $\left\Vert \mathcal{B}^{n}f\right\Vert _{%
\mathcal{F}(%
%TCIMACRO{\U{2102} }%
%BeginExpansion
\mathbb{C}
%EndExpansion
)}=\left\Vert f\right\Vert _{L^{2}(%
%TCIMACRO{\U{211d} }%
%BeginExpansion
\mathbb{R}
%EndExpansion
)}$. See also \cite{Hunik} for an alternative approach.\ Furthermore, the
range $\cB^{n}(L^{2}(\mathbb{R}))$ under $\mathcal{B}^{n}$ consists exactly
of all polyanalytic functions $F$ satisfying the integrability condition $%
\int_{\bC}|F(z)|^{2}e^{-\pi |z|^{2}}\,dz<\infty $ and such that, for some
entire function $H$,\ 
\begin{equation*}
F(z)=\left( \frac{\pi ^{n}}{n!}\right) ^{\frac{1}{2}}e^{-\pi \left\vert
z\right\vert ^{2}}\left( \frac{d}{dz}\right) ^{n}\left[ e^{-\pi \left\vert
z\right\vert ^{2}}H(z)\right] \text{.}
\end{equation*}%
%
%
%
%  with membership in a space equiped
% with the same norm as $\mathcal{F}(%
% %TCIMACRO{\U{2102} }%
% %BeginExpansion
% \mathbb{C}
% %EndExpansion
% )$, whose elements are not analytic functions but satisfy the equation (\ref%
% {polyanalytic}).
We denote the range of $\cB^{n}$ by $\mathcal{F}^{n}(%
%TCIMACRO{\U{2102} }%
%BeginExpansion
\mathbb{C}
%EndExpansion
)=\cB(L^{2}(\bR))$ and call $\cF^{n}(\bC)$ the true polyanalytic Fock space
of order $n$, %  ( their
% elements are polyanalytic of order $n$, but not polyanalytic of any other
% lower order,
see \cite{Abreusampling}, \cite{Abreustructure} and \cite{VasiFock}. The
prefix "true" has been used by Vasilevski \cite{VasiFock} to distinguish
them from the polyanalytic Fock space $\mathbf{F}^{n}(%
%TCIMACRO{\U{2102} }%
%BeginExpansion
\mathbb{C}
%EndExpansion
)$, which consists of \emph{all }polyanalytic functions up to order $n$.
This space possesses the orthogonal decomposition:%
\begin{equation}
\mathbf{F}^{n}(%
%TCIMACRO{\U{2102} }%
%BeginExpansion
\mathbb{C}
%EndExpansion
)=\mathcal{F}^{1}(%
%TCIMACRO{\U{2102} }%
%BeginExpansion
\mathbb{C}
%EndExpansion
)\oplus ...\oplus \mathcal{F}^{n}(%
%TCIMACRO{\U{2102} }%
%BeginExpansion
\mathbb{C}
%EndExpansion
),  \label{orthogonal}
\end{equation}%
We remark that the orthogonality of the $\cF^{k}(\bC)$ follows directly from
the orthogonality relations of the \stft\ and \eqref{polyBargmann}. The
spaces $\mathcal{F}^{n}(%
%TCIMACRO{\U{2102} }%
%BeginExpansion
\mathbb{C}
%EndExpansion
)$ also appear in connection to the eigenspaces of an operator related to
the Landau levels \cite{AAZ}.

The orthogonal decomposition ~(\ref{orthogonal}) suggests a second
transform. Define $\mathbf{B}^{n}:L^{2}(%
%TCIMACRO{\U{211d} }%
%BeginExpansion
\mathbb{R}
%EndExpansion
,%
%TCIMACRO{\U{2102} }%
%BeginExpansion
\mathbb{C}
%EndExpansion
^{n})\rightarrow \mathbf{F}^{n}(%
%TCIMACRO{\U{2102} }%
%BeginExpansion
\mathbb{C}
%EndExpansion
)$ by mapping each vector $\mathbf{f=(}f_{1},...,f_{n}\mathbf{)}\in L^{2}(%
%TCIMACRO{\U{211d} }%
%BeginExpansion
\mathbb{R}
%EndExpansion
,%
%TCIMACRO{\U{2102} }%
%BeginExpansion
\mathbb{C}
%EndExpansion
^{n})$ to the following polyanalytic function of order $n$: 
\begin{equation}
\mathbf{B}^{n}\mathbf{f=}\mathcal{B}^{1}f_{1}+...+\mathcal{B}^{n}f_{n}\,.
\label{polyanalyticBargmann}
\end{equation}

This map is again a Hilbert space isomorphism and is called the \emph{%
polyanalytic Bargmann transform }\cite{Abreusampling}\emph{. }

The polyanalytic Bargmann transform possesses an interesting interpretation
in signal processing.

\begin{enumerate}
\item Given $n$ signals $f_{1}, \dots ,f_{n}$, with finite energy ($f_{k}\in
L^{2}(%
%TCIMACRO{\U{211d} }%
%BeginExpansion
\mathbb{R}
%EndExpansion
)$ for every $k$), process each individual signal by evaluating $\mathcal{B}%
^{k}f_{k}$. This encodes each signal into one of the $n$ orthogonal spaces $%
\mathcal{F}^{1}(%
%TCIMACRO{\U{2102} }%
%BeginExpansion
\mathbb{C}
%EndExpansion
), \dots ,\mathcal{F}^{n}(%
%TCIMACRO{\U{2102} }%
%BeginExpansion
\mathbb{C}
%EndExpansion
)$.

\item Construct a new signal $F= \mathbf{B}\mathbf{f} =\mathcal{B}%
^{1}f_{1}+...+\mathcal{B}^{n}f_{n} $ as a superposition of the $n$ processed
signals.

\item Sample, transmit, or process $F.$

\item Let $P^{k}$ denote the orthogonal projection from $\mathbf{F}^{n}(%
%TCIMACRO{\U{2102} }%
%BeginExpansion
\mathbb{C}
%EndExpansion
)$ onto $\mathcal{F}^{k}(%
%TCIMACRO{\U{2102} }%
%BeginExpansion
\mathbb{C}
%EndExpansion
)$, then $P^{k}\left( F\right) =\mathcal{B}^{k}f_{k}$ by virtue of (\ref%
{orthogonal}).

\item Finally, after inverting each of the transforms $\mathcal{B}^{k}$, we
recover each component $f_{k}$ in its original form.
\end{enumerate}

The combination of $n$ independent signals into a single signal $\mathbf{B}%
^n \mathbf{f}$ and the subsequent processing are referred to as multiplexing.

The projection $P^{k}$ can be written explicitly as an integral operator,
see Proposition~3 and \cite{Abreustructure}.

\section{$L^{p}$ theory of polyanalytic Fock spaces}

\subsection{The spaces $\mathbf{F}_{p}^{n}$ and $\mathcal{F}_{p}^{n}$}

We next introduce the $L^{p}$ version of the polyanalytic Bargmann-Fock
spaces. For\ $p\in \lbrack 1,\infty \lbrack $ write $\mathcal{L}_{p}(%
%TCIMACRO{\U{2102} }%
%BeginExpansion
\mathbb{C}
%EndExpansion
)$ to denote the Banach space of all measurable functions equiped with the
norm 
\begin{equation*}
\left\Vert F\right\Vert _{\mathcal{L}_{p}(%
%TCIMACRO{\U{2102} }%
%BeginExpansion
\mathbb{C}
%EndExpansion
)}=\Big(\int_{%
%TCIMACRO{\U{2102} }%
%BeginExpansion
\mathbb{C}
%EndExpansion
}\left\vert F(z)\right\vert ^{p}e^{-\pi p\frac{\left\vert z\right\vert }{2}%
^{2}}\, dz \Big) ^{1/p} \text{.}
\end{equation*}%
For $p=\infty $, we have $\left\Vert F\right\Vert _{\mathcal{L}_{\infty }(%
%TCIMACRO{\U{2102} }%
%BeginExpansion
\mathbb{C}
%EndExpansion
)}=\sup_{z\in 
%TCIMACRO{\U{2102} }%
%BeginExpansion
\mathbb{C}
%EndExpansion
}\left\vert F(z)\right\vert e^{-\pi \frac{\left\vert z\right\vert ^{2}}{2}}.$
With this notation we extend the definitions of polyanalytic Fock spaces to
the Banach space setting.

\begin{definition}
We say that a function $F$ belongs to the \emph{polyanalytic Fock space} $%
\mathbf{F}_{p}^{n}\left( 
%TCIMACRO{\U{2102} }%
%BeginExpansion
\mathbb{C}
%EndExpansion
\right) $, if $\left\Vert F\right\Vert _{\mathcal{L}_{p}(%
%TCIMACRO{\U{2102} }%
%BeginExpansion
\mathbb{C}
%EndExpansion
)}<\infty $ and $F$ is polyanalytic of order $n.$
\end{definition}

\begin{definition}
We say that a function $F$ belongs to the \emph{true polyanalytic Fock space}
$\mathcal{F}_{p}^{n+1}(%
%TCIMACRO{\U{2102} }%
%BeginExpansion
\mathbb{C}
%EndExpansion
)$ if $\left\Vert F\right\Vert _{\mathcal{L}_{p}(%
%TCIMACRO{\U{2102} }%
%BeginExpansion
\mathbb{C}
%EndExpansion
)}<\infty $ and there exists an entire function $H$ such that 
\begin{equation*}
F(z)=\left( \frac{\pi ^{n}}{n!}\right) ^{\frac{1}{2}}e^{-\pi \left\vert
z\right\vert ^{2}}\left( \frac{d}{dz}\right) ^{n}\left[ e^{-\pi \left\vert
z\right\vert ^{2}}H(z)\right] .
\end{equation*}
\end{definition}

Clearly, $\mathcal{F}_{p}^{1}(%
%TCIMACRO{\U{2102} }%
%BeginExpansion
\mathbb{C}
%EndExpansion
)=\mathcal{F}_{p}(%
%TCIMACRO{\U{2102} }%
%BeginExpansion
\mathbb{C}
%EndExpansion
)$ is the standard Bargmann-Fock space. The space $\mathcal{F}_{1}^1(%
%TCIMACRO{\U{2102} }%
%BeginExpansion
\mathbb{C}
%EndExpansion
)$ is the complex version of the Feichtinger algebra \cite{Fei,FeiZim}, and
it will play an important role in last section of the paper.

\subsection{Orthogonal decompositions}

In dealing with polyanalytic functions the following version of integration
by parts is useful. The disc of radius $r$ is denoted by $\mathbf{D}_{r}$ as
usual, its boundary is $\delta \mathbf{D}_{r}$.

\begin{lemma}
\label{intpart} If $f,g\in C^{1}(\mathbf{D}_{r})$, then 
\begin{equation}
\int_{\mathbf{D}_{r}}f(z)\frac{d}{d\overline{z}}\overline{g(z)}dz=-\int_{%
\mathbf{D}_{r}}\frac{d}{d\overline{z}}f(z)\overline{g(z)}dz+\frac{1}{i}%
\int_{\delta \mathbf{D}_{r}}f(z)\overline{g(z)}dz\text{,}  \label{Green_1}
\end{equation}%
where \ the line integral over the circle $\delta \mathbf{D}_{r}$ is
oriented counterclockwise.
\end{lemma}

Iterating (\ref{Green_1}) one obtains the formula 
\begin{eqnarray}
\int_{\mathbf{D}_{r}}f(z)\left( \frac{d}{d\overline{z}}\right) ^{n}\overline{%
g(z)}dz &=&(-1)^{n}\int_{\mathbf{D}_{r}}\left( \frac{d}{d\overline{z}}%
\right) ^{n}f(z)\overline{g(z)}dz  \notag \\
&&+\frac{1}{i}\sum_{j=0}^{n-1}(-1)^{j}\int_{\delta \mathbf{D}_{r}}\left( 
\frac{d}{d\overline{z}}\right) ^{j}f(z)\left( \frac{d}{d\overline{z}}\right)
^{n-j-1}\overline{g(z)}dz\text{.}  \label{Green_n}
\end{eqnarray}%
The Lemma follows from Green%
%TCIMACRO{\U{b4}}%
%BeginExpansion
\'{}%
%EndExpansion
s formula. It can also be seen directly for polyanalytic polynomials $p(z,%
\overline{z})$ and then extended by density.

Now recall that the monomials 
\begin{equation*}
e_{n}(z)=\left( \frac{\pi ^{n}}{n!}\right) ^{\frac{1}{2}}z^{n}
\end{equation*}%
provide an orthonormal basis for $\mathcal{F}_{2}(%
%TCIMACRO{\U{2102} }%
%BeginExpansion
\mathbb{C}
%EndExpansion
)$ and that $\mathcal{B(}h_{n})\mathcal{=}e_{n}$. In addition, they are
orthogonal on every disk $D_{r}$: for every $r>0$,%
\begin{equation}
\int_{D_{r}}e_{n}(z)\overline{e_{m}(z)}e^{-\pi \frac{\left\vert z\right\vert
^{2}}{2}}dz= C(r,m) \, \delta _{nm}\text{.}  \label{ortdisk}
\end{equation}%
The normalization constant $C(r,m)$ depends on $r$ and $m$ and satisfies $%
\lim _{r\to \infty } C(r,m) = 1$. The double orthogonality follows easily by
writing the integral in polar coordinates (see \cite[pg. 54]{Charly}). It
will be the key fact behind the proof of the next result.

\begin{proposition}
\label{densityhm} The sequence of polyanalytic orthogonal polynomials
defined as 
\begin{equation*}
e_{k,m}(z)=e^{\pi \left\vert z\right\vert ^{2}}\left( \frac{d}{dz}\right)
^{k}\left[ e^{-\pi \left\vert z\right\vert ^{2}}e_{m}(z)\right] \text{,}
\end{equation*}%
enjoy the following properties:

\begin{enumerate}
\item The linear span of $\{e_{k,m} : m\geq 0,0\leq k\leq n\}$ is dense in $%
\mathbf{F}_{p}^{n+1}(%
%TCIMACRO{\U{2102} }%
%BeginExpansion
\mathbb{C}
%EndExpansion
)$ for $1\leq p < \infty $ (and weak-$^*$ dense in $\mathbf{F} ^{n+1}_\infty
(\bC )$.

\item For fixed $n$ the linear span of $\{e_{n,m}: m\geq 0 \}$ is dense in $%
\mathcal{F}_{p}^{n+1}(%
%TCIMACRO{\U{2102} }%
%BeginExpansion
\mathbb{C}
%EndExpansion
)$ for $1\leq p < \infty $.

\item $\mathcal{B}^{k}\mathcal{(}h_{m})\mathcal{=}\left( \frac{\pi ^{m}}{m!}%
\right) ^{\frac{1}{2}}e_{k,m}.$
\end{enumerate}
\end{proposition}

\begin{proof}
To prove completeness of $\{e_{k,m}\}$ in $\mathbf{F}_{p}^{n}(%
%TCIMACRO{\U{2102} }%
%BeginExpansion
\mathbb{C}
%EndExpansion
)$, suppose that $F\in \mathbf{F}_{p}^{n}(%
%TCIMACRO{\U{2102} }%
%BeginExpansion
\mathbb{C}
%EndExpansion
)$ is such that $\left\langle F,e_{k,m}\right\rangle _{\mathcal{L}_{p}(%
%TCIMACRO{\U{2102} }%
%BeginExpansion
\mathbb{C}
%EndExpansion
)}=0$,\ for all\ $0\leq k\leq n-1$ and\ \ $m\geq 0.$ For $n=1$, $k=0$, the
classical argument for the completeness of the monomials shows that $F=0$.
For $k\geq 1$, we use formula (\ref{Green_n}) as follows: 
\begin{eqnarray*}
&&\int_{\mathbf{D}_{r}}F(z)\overline{e_{k,m}(z)}e^{-\pi \left\vert
z\right\vert ^{2}}dz=\int_{\mathbf{D}_{r}}F(z)e^{\pi \left\vert z\right\vert
^{2}}\left( \frac{d}{d\overline{z}}\right) ^{k}\left[ e^{-\pi \left\vert
z\right\vert ^{2}}\overline{e_{m}(z)}\right] e^{-\pi \left\vert z\right\vert
^{2}}dz \\
&=&\int_{\mathbf{D}_{r}}F(z)\left( \frac{d}{d\overline{z}}\right) ^{k}\left[
e^{-\pi \left\vert z\right\vert ^{2}}\overline{e_{m}(z)}\right] dz \\
&=&(-1)^{k}\int_{\mathbf{D}_{r}}\left( \frac{d}{d\overline{z}}\right)
^{k}F(z)\left[ e^{-\pi \left\vert z\right\vert ^{2}}\overline{e_{m}(z)}%
\right] dz \\
&&+\frac{1}{i}\sum_{j=0}^{n-1}(-1)^{j}\int_{\delta \mathbf{D}_{r}}\left( 
\frac{d}{d\overline{z}}\right) ^{j}F(z)\left( \frac{d}{d\overline{z}}\right)
^{k-j-1}\left[ e^{-\pi \left\vert z\right\vert ^{2}}\overline{e_{m}(z)}%
\right] dz\text{.}
\end{eqnarray*}%
Now, using the representation (\ref{polypolynomial}) we can write the
polyanalytic function $F$ in the form:%
\begin{equation*}
F(z)=\sum_{0\leq p\leq n-1}\overline{z}^{p}\sum_{l\geq 0}c_{l,p}z^{l}\text{.}
\end{equation*}%
Since the sum converges uniformly over compact set, we interchange the order
of summation and integration in the following manipulations. 
\begin{eqnarray*}
&&\int_{\delta \mathbf{D}_{r}}\left( \frac{d}{d\overline{z}}\right)
^{j}F(z)\left( \frac{d}{d\overline{z}}\right) ^{k-j-1}\left[ e^{-\pi
\left\vert z\right\vert ^{2}}\overline{e_{m}(z)}\right] dz \\
&=&\int_{\delta \mathbf{D}_{r}}\sum_{j\leq p\leq n-1}p...(p-j+1)\overline{z}%
^{p-j}\sum_{l\geq 0}c_{l,p}z^{l}\left( \frac{d}{d\overline{z}}\right)
^{k-j-1}\left[ e^{-\pi \left\vert z\right\vert ^{2}}\overline{e_{m}(z)}%
\right] dz \\
&=&\sum_{j\leq p\leq n-1}p...(p-j+1)\overline{z}^{p-j}\sum_{l\geq
0}c_{l,p}\,\int_{\delta \mathbf{D}_{r}}\overline{z}^{p-k}z^{j}\left( \frac{d%
}{d\overline{z}}\right) ^{k-j-1}\left[ e^{-\pi \left\vert z\right\vert ^{2}}%
\overline{e_{m}(z)}\right] dz
\end{eqnarray*}%
If we let $r\rightarrow \infty $, the integral on the last expression
approaches zero and since the function  $\varphi _{p}(z)=\sum_{l\geq
0}c_{l,p}z^{l}$ is analytic, the coefficients $\{c_{l,p}\}_{l\geq 0}$ decay
fast enough in order to assure that the whole expression approaches zero.
Thus,%
\begin{eqnarray*}
\int_{\mathbf{D}_{r}}F(z)\overline{e_{k,m}(z)}e^{-\pi \left\vert
z\right\vert ^{2}}dz &=&(-1)^{k}\int_{\mathbf{D}_{r}}\left( \frac{d}{d%
\overline{z}}\right) ^{k}F(z)\left[ e^{-\pi \left\vert z\right\vert ^{2}}%
\overline{e_{m}(z)}\right] dz \\
&=&(-1)^{k}\sum_{k\leq p\leq n}\sum_{j\geq 0}c_{j,p}\frac{p...(p-k+1)\pi
^{m/2}}{\sqrt{m!}}\int_{\mathbf{D}_{r}}z^{j}\overline{z}^{m+p-k}e^{-\pi
\left\vert z\right\vert ^{2}}dz.
\end{eqnarray*}%
We first use \ (\ref{ortdisk})\ and then let $r\rightarrow \infty $. The
hypothesis $0=\left\langle F,e_{k,m}\right\rangle _{\mathcal{F}(%
%TCIMACRO{\U{2102} }%
%BeginExpansion
\mathbb{C}
%EndExpansion
^{d})}$ for $0\leq k\leq n$ implies that 
\begin{equation*}
\sum_{k\leq p\leq n}\frac{p...(p-k+1)(p+m-k)!}{\pi ^{\frac{3}{2}+p-k}\sqrt{m!%
}}c_{m+p-k,p}=0\text{, \ \ }m\geq 0\text{, }0\leq k\leq n\text{.}
\end{equation*}
Solving the resulting triangular system for each $m$, we obtain $c_{j,p}=0$
for $k\leq p\leq n-1$ and $j\geq 0$. Therefore $F=0$.

To prove item (2), suppose now that $F\in \mathcal{F}_{p}^{n+1}(%
%TCIMACRO{\U{2102} }%
%BeginExpansion
\mathbb{C}
%EndExpansion
)$. Then there exists an entire function $H(z)=\sum_{j\geq 0}a_{j}z^{j}$
such that 
\begin{equation*}
F(z)=\left( \frac{\pi ^{n}}{n!}\right) ^{\frac{1}{2}}e^{-\pi \left\vert
z\right\vert ^{2}}\left( \frac{d}{dz}\right) ^{n}\left[ e^{-\pi \left\vert
z\right\vert ^{2}}H(z)\right] =\left( \frac{\pi ^{n}}{n!}\right) ^{\frac{1}{2%
}}\sum_{0\leq k\leq n}\binom{n}{k}\left( -\pi \overline{z}\right) ^{k}\left( 
\frac{d}{dz}\right) ^{n-k}H(z).
\end{equation*}%
We apply Lemma~\eqref{intpart} and denote by $B(r)$ the boundary terms
arising from (\ref{Green_n}). Then 
\begin{eqnarray*}
\int_{\mathbf{D}_{r}}F(z)\overline{e_{n,m}(z)}e^{-\pi \left\vert
z\right\vert ^{2}}dz &=&\int_{\mathbf{D}_{r}}F(z)\left( \frac{d}{d\overline{z%
}}\right) ^{n}\left[ e^{-\pi \left\vert z\right\vert ^{2}}\overline{e_{m}(z)}%
\right] dz+B(r) \\
&=&(-1)^{n}\int_{\mathbf{D}_{r}}\left( \frac{d}{d\overline{z}}\right)
^{n}F(z)\overline{e_{m}(z)}e^{-\pi \left\vert z\right\vert ^{2}}dz+B(r) \\
&=&(-1)^{n}\int_{\mathbf{D}_{r}}H(z)\overline{e_{m}(z)}e^{-\pi \left\vert
z\right\vert ^{2}}dz+B(r) \\
&=&\sum_{j\geq 0}a_{j}\int_{\mathbf{D}_{r}}z^{j}\overline{e_{m}(z)}e^{-\pi
\left\vert z\right\vert ^{2}}dz+B(r) \\
&=&C(r,m)a_{m}+B(r)\,,
\end{eqnarray*}%
where we have used (\ref{ortdisk}) in the last equality. If $r\rightarrow
\infty $, then $B(r)\rightarrow 0$, and the hypothesis $0=\left\langle
F,e_{n,m}\right\rangle _{\mathcal{F}(%
%TCIMACRO{\U{2102} }%
%BeginExpansion
\mathbb{C}
%EndExpansion
)},m\geq 0$\ now implies that $a_{j}=0,j\geq 0$. Thus $F=0$.

Assertion (3) of the proposition is an immediate consequence of the
definition of the true polyanalytic Bargmann transform and the fact that $%
\mathcal{B(}h_{m})\mathcal{=}e_{m}$ (see also \cite{Abreusampling}).
\end{proof}

An obvious consequence of the above proposition is the extension of the
orthogonal decomposition (\ref{orthogonal}) to the $p$-norm setting. Similar
results appear in \cite{Ramazanov} for the unit disk case. See also \cite%
{Begehr} for other approaches to decompositions in spaces of polyanalytic
functions.

\begin{corollary}
The following decompositions hold for $1<p<\infty $: 
\begin{eqnarray*}
\mathbf{F}_{p}^{n}(%
%TCIMACRO{\U{2102} }%
%BeginExpansion
\mathbb{C}
%EndExpansion
) &=&\mathcal{F}_{p}^{1}(%
%TCIMACRO{\U{2102} }%
%BeginExpansion
\mathbb{C}
%EndExpansion
)\oplus ...\oplus \mathcal{F}_{p}^{n}(%
%TCIMACRO{\U{2102} }%
%BeginExpansion
\mathbb{C}
%EndExpansion
). \\
\mathcal{L}_{p}(%
%TCIMACRO{\U{2102} }%
%BeginExpansion
\mathbb{C}
%EndExpansion
) &=&\bigoplus_{n=1}^{\infty }\mathcal{F}_{p}^{n}(%
%TCIMACRO{\U{2102} }%
%BeginExpansion
\mathbb{C}
%EndExpansion
).
\end{eqnarray*}
\end{corollary}

\subsection{\protect\bigskip Mapping properties of the true polyanalytic
Bargmann transform in modulation spaces}

For the investigation of the mapping properties of the true polyanalytic
Bargmann transform $\mathcal{F}_{p}^{n}(%
%TCIMACRO{\U{2102} }%
%BeginExpansion
\mathbb{C}
%EndExpansion
)$ we need the concept of \emph{modulation space}. Following \cite{Charly},
the modulation space $M^{p}(%
%TCIMACRO{\U{211d} }%
%BeginExpansion
\mathbb{R}
%EndExpansion
),$ $1\leq p\leq \infty ,$ consists of all tempered distributions $f$ such
that $V_{h_{0}}f\in L^{p}(%
%TCIMACRO{\U{211d} }%
%BeginExpansion
\mathbb{R}
%EndExpansion
^{2})$ equipped with the norm%
\begin{equation*}
\left\Vert f\right\Vert _{M^{p}(%
%TCIMACRO{\U{211d} }%
%BeginExpansion
\mathbb{R}
%EndExpansion
)}=\left\Vert V_{h_{0}}f\right\Vert _{L^{p}(%
%TCIMACRO{\U{211d} }%
%BeginExpansion
\mathbb{R}
%EndExpansion
^{2})}\text{.}
\end{equation*}%
Modulation spaces are ubiquitous in time-frequency analysis. They were
introduced by Feichtinger in \cite{FeiModulation}.

With a view to studying sampling sequences in poly-Fock spaces $\mathcal{F}%
_{p}^{n}(%
%TCIMACRO{\U{2102} }%
%BeginExpansion
\mathbb{C}
%EndExpansion
)$\ for general $p$, we prove some statements concerning the properties of
the true poly-Bargmann transform, which may be of independent interest.

\begin{proposition}
There exist constants $C,D$, such that, for every $f\in M^{p}(%
%TCIMACRO{\U{211d} }%
%BeginExpansion
\mathbb{R}
%EndExpansion
)$, $1\leq p \leq \infty $, 
\begin{equation}
C\left\Vert \mathcal{B}^{n}f\right\Vert _{\mathcal{L}_{p}(%
%TCIMACRO{\U{2102} }%
%BeginExpansion
\mathbb{C}
%EndExpansion
)}\leq \left\Vert \mathcal{B}f\right\Vert _{\mathcal{L}_{p}(%
%TCIMACRO{\U{2102} }%
%BeginExpansion
\mathbb{C}
%EndExpansion
)}\leq D\left\Vert \mathcal{B}^{n}f\right\Vert _{\mathcal{L}_{p}(%
%TCIMACRO{\U{2102} }%
%BeginExpansion
\mathbb{C}
%EndExpansion
)}\text{.}  \label{constants}
\end{equation}
\end{proposition}

\begin{proof}
This follows from the theory of modulation spaces: since the definition of
Modulation space is independent of the particular window chosen \cite[%
Proposition 11.3.1]{Charly}, then the norms%
\begin{equation*}
\left\Vert f\right\Vert ^{\prime}_{M^{p}(%
%TCIMACRO{\U{211d} }%
%BeginExpansion
\mathbb{R}
%EndExpansion
^{2})}=\left\Vert V_{h_{n}}f\right\Vert _{L^{p}(%
%TCIMACRO{\U{211d} }%
%BeginExpansion
\mathbb{R}
%EndExpansion
^{2})}
\end{equation*}%
and%
\begin{equation*}
\left\Vert f\right\Vert _{M^{p}(%
%TCIMACRO{\U{211d} }%
%BeginExpansion
\mathbb{R}
%EndExpansion
^{2})}=\left\Vert V_{h_{0}}f\right\Vert _{L^{p}(%
%TCIMACRO{\U{211d} }%
%BeginExpansion
\mathbb{R}
%EndExpansion
^{2})}\text{,}
\end{equation*}%
are equivalent. Therefore, there exist constants $C,D$, such that%
\begin{equation*}
C\left\Vert V_{h_{n}}f\right\Vert _{L^{p}(%
%TCIMACRO{\U{211d} }%
%BeginExpansion
\mathbb{R}
%EndExpansion
^{2})}\leq \left\Vert V_{h_{0}}f\right\Vert _{\mathcal{F}_{p}(%
%TCIMACRO{\U{2102} }%
%BeginExpansion
\mathbb{C}
%EndExpansion
)}\leq D\left\Vert V_{h_{n}}f\right\Vert _{L^{p}(%
%TCIMACRO{\U{211d} }%
%BeginExpansion
\mathbb{R}
%EndExpansion
^{2})}\text{.}
\end{equation*}%
By definition of $\mathcal{B}^{n}$ and $\mathcal{B}$, this yields (\ref%
{constants}).
\end{proof}

The next result includes the surjectivity of the transform $\mathcal{B}^{n}$
onto $\mathcal{F}_{p}^{n}(%
%TCIMACRO{\U{2102} }%
%BeginExpansion
\mathbb{C}
%EndExpansion
).$

\begin{corollary}
Given $F\in \mathcal{F}_{p}^{n}(%
%TCIMACRO{\U{2102} }%
%BeginExpansion
\mathbb{C}
%EndExpansion
)$ there exists $f\in M^{p}(%
%TCIMACRO{\U{211d} }%
%BeginExpansion
\mathbb{R}
%EndExpansion
)$ such that $F=\mathcal{B}^{n}f$. Moreover, there exist constants $C,D$
such that: 
\begin{equation}
C\left\Vert F\right\Vert _{\mathcal{L}_{p}(%
%TCIMACRO{\U{2102} }%
%BeginExpansion
\mathbb{C}
%EndExpansion
)}\leq \left\Vert \mathcal{B}f\right\Vert _{\mathcal{L}_{p}(%
%TCIMACRO{\U{2102} }%
%BeginExpansion
\mathbb{C}
%EndExpansion
)}\leq D\left\Vert F\right\Vert _{\mathcal{L}_{p}(%
%TCIMACRO{\U{2102} }%
%BeginExpansion
\mathbb{C}
%EndExpansion
)}.  \label{constants_2}
\end{equation}
\end{corollary}

\begin{proof}
Since the Hermite functions belong to $M^{p}$\ and $\mathcal{B}%
^{n}(h_{k})=e_{k,n}$, the range of $\mathcal{B}^{n}$ contains a set which is
dense in $\mathcal{F}_{p}^{n}(%
%TCIMACRO{\U{2102} }%
%BeginExpansion
\mathbb{C}
%EndExpansion
)$ for $1\leq p < \infty$. Thus, $\mathcal{B}^{n}:L^{2}(%
%TCIMACRO{\U{211d} }%
%BeginExpansion
\mathbb{R}
%EndExpansion
)\rightarrow \mathcal{F}_{p}^{n}(%
%TCIMACRO{\U{2102} }%
%BeginExpansion
\mathbb{C}
%EndExpansion
)$ is onto for $1\leq p < \infty$. Then (\ref{constants}) is equivalent to (%
\ref{constants_2}).

For $p=\infty $ we use that the span of the $e_{k,m}$ is weak-$^*$ dense in $%
\cF ^n_\infty (\bC )$.
\end{proof}

\subsection{The polyanalytic projection}

Let 
\begin{equation*}
K^{n}(w,z)=\frac{1}{n!}e^{\pi \left\vert w\right\vert ^{2}}\left( \frac{d}{dw%
}\right) ^{n}\left[ e^{\pi \overline{z}w-\pi \left\vert w\right\vert
^{2}}(w-z)^{n}\right]
\end{equation*}%
denote the reproducing kernel of $\mathcal{F}_{2}^{n+1}(%
%TCIMACRO{\U{2102} }%
%BeginExpansion
\mathbb{C}
%EndExpansion
)$ and define the integral operator $P^{n}$ by 
\begin{equation*}
(P^{n}F)(w)=\int_{%
%TCIMACRO{\U{2102} }%
%BeginExpansion
\mathbb{C}
%EndExpansion
}F(z)K^{n}(w,z)e^{-\pi \left\vert z\right\vert ^{2}}dz\text{.}
\end{equation*}

We have shown in~\cite{Abreustructure} that $P^{n}$ is the orthogonal
projection from $\mathcal{L}_{2}(%
%TCIMACRO{\U{2102} }%
%BeginExpansion
\mathbb{C}
%EndExpansion
)$ onto the true polyanalytic Fock space $\mathcal{F}_{2}^{n+1}$. We now
show that $P^{n}$ is also bounded on $\mathcal{L}_{p}(%
%TCIMACRO{\U{2102} }%
%BeginExpansion
\mathbb{C}
%EndExpansion
)$ and extend the reproducing property to $\mathcal{F}_{p}^{n+1}$.

\begin{proposition}
The operator $P^{n}$ is bounded from $\mathcal{L}_{p}(%
%TCIMACRO{\U{2102} }%
%BeginExpansion
\mathbb{C}
%EndExpansion
)$ to $\mathcal{F}_{p}^{n+1}$ for $1\leq p\leq \infty $. Moreover, if $F\in 
\mathcal{F}_{p}^{n+1}$ then $P^{n}F=F.$
\end{proposition}

\begin{proof}
The kernel $K^{n}$ is 
\begin{equation*}
K^{n}(w,z)=\frac{1}{n!}e^{\pi \left\vert w\right\vert ^{2}}\left( \frac{d}{dw%
}\right) ^{n}\left[ e^{\pi \overline{z}w-\pi \left\vert w\right\vert
^{2}}(w-z)^{n}\right] =\sum_{k=0}^{n}\binom{n}{k}\frac{1}{k!}(-\pi
|w-z|^{2})^{k}e^{\pi \bar{z}w}\,,
\end{equation*}%
so 
\begin{equation*}
P^{n}F(w)e^{-\pi |w|^{2}/2}=\int_{%
%TCIMACRO{\U{2102} }%
%BeginExpansion
\mathbb{C}
%EndExpansion
}F(z)e^{-\pi |z|^{2}/2}\left( \sum_{k=0}^{n}\binom{n}{k}\frac{1}{k!}(-\pi
|w-z|^{2})^{k}e^{\pi \bar{z}w}\right) e^{-\pi |w|^{2}/2}\,e^{-\pi
|z|^{2}/2}\,dz\,.
\end{equation*}%
We take absolute values and observe that $e^{-\pi |w-z|^{2}/2}=|e^{\pi \bar{z%
}w}|e^{-\pi |w|^{2}/2}e^{-\pi |z|^{2}/2}$, in this way we obtain that 
\begin{equation*}
|P^{n}F(w)e^{-\pi |w|^{2}/2}|\leq \int_{%
%TCIMACRO{\U{2102} }%
%BeginExpansion
\mathbb{C}
%EndExpansion
}|F(z)|e^{-\pi |z|^{2}/2}\left( \sum_{k=0}^{n}\binom{n}{k}\frac{1}{k!}(-\pi
|w-z|^{2})^{k}\right) e^{-\pi |w-z|^{2}/2}\,dz\, .
\end{equation*}%
Now set $\Phi (z)=|F(z)|e^{-\pi |z|^{2}/2}\in L^{p}(%
%TCIMACRO{\U{211d} }%
%BeginExpansion
\mathbb{R}
%EndExpansion
^{2})$ and $\Psi (z)=\left( \sum_{k=0}^{n}\binom{n}{k}\frac{1}{k!}(-\pi
|w-z|^{2})^{k}\right) e^{-\pi |w-z|^{2}/2}\in L^{1}(%
%TCIMACRO{\U{211d} }%
%BeginExpansion
\mathbb{R}
%EndExpansion
^{2})$. Then 
\begin{align}
\left( \int_{%
%TCIMACRO{\U{2102} }%
%BeginExpansion
\mathbb{C}
%EndExpansion
}|P^{n}F(w)|^{p}e^{-\pi p|w|^{2}/2}\,dw\right) ^{1/p}& \leq \Vert \Phi \ast
\Psi \Vert _{p}  \notag \\
& \leq \Vert \Phi \Vert _{p}\,\Vert \Psi \Vert _{1}  \notag \\
& =C\left( \int_{%
%TCIMACRO{\U{2102} }%
%BeginExpansion
\mathbb{C}
%EndExpansion
}|F(z)|^{p}e^{-\pi p|z|^{2}/2}\,dz\right) ^{1/p}=\Vert F\Vert _{\mathcal{L}%
_{p}}\,,  \label{boundt}
\end{align}%
and thus $P^{n}$ is bounded on $\mathcal{L}_{p}$.

Next set $H(w) =\int_{%
%TCIMACRO{\U{2102} }%
%BeginExpansion
\mathbb{C}
%EndExpansion
}F(z)(w-z)^{n}e^{\pi \overline{z}w}\,e^{-\pi |z|^{2}}\,dz$. Since $F\in \cL %
_p$, the integral is well-defined and $H$ is an entire function. Since 
\begin{equation*}
P^{n}F(w)=\frac{1}{n!}e^{\pi \left\vert w\right\vert ^{2}}\left( \frac{d}{dw}%
\right) ^{n}\left( e^{-\pi |w|^{2}}\,\int_{%
%TCIMACRO{\U{2102} }%
%BeginExpansion
\mathbb{C}
%EndExpansion
}F(z)(w-z)^{n}e^{\pi \overline{z}w}\,e^{-\pi |z|^{2}}\,dz\,\right) \,,
\end{equation*}%
it follows that $P^{n}F$ is a true polyanalytic function. By the boundedness
of $P^n$, we also have $P^{n}F\in \mathcal{F}_{p}^{n+1}$. Finally, if $F \in 
\mathcal{F}_{2}^{n+1}\cap \mathcal{F}_{p}^{n+1}$, then $P^nF = F$ by the
reproducing kernel property in $\mathcal{F}_{2}^{n+1}$. Since $\mathcal{F}%
_{2}^{n+1}\cap \mathcal{F}_{p}^{n+1}$ is dense in $\mathcal{F}_{p}^{n+1}$ by
Lemma~\ref{densityhm}, the identity $P^nF = F$ extends to all $F\in \cF %
^{n+1}_p$.
\end{proof}

\section{Gabor frames in $L^{2}$}

Stable Gabor expansions of the form (\ref{HermiteExpansions}) can be
obtained from frame theory. Given a point $\lambda =(\lambda _1,\lambda _2 )$
in phase-space $%
%TCIMACRO{\U{211d} }%
%BeginExpansion
\mathbb{R}
%EndExpansion
^{2}$, the corresponding time-frequency shift is%
\begin{equation*}
\pi _{\lambda }f(t)=e^{2\pi i\lambda _2 t}f(t-\lambda _1)\text{, \ \ \ }t\in 
%TCIMACRO{\U{211d} }%
%BeginExpansion
\mathbb{R}
%EndExpansion
\text{.}
\end{equation*}%
Using this notation, the \stft\ of a function $f$ with respect to the window 
$g$ can be written as%
\begin{equation*}
V_{g}f(\lambda )=\left\langle f,\pi _{\lambda }g\right\rangle _{L^{2}(%
%TCIMACRO{\U{211d} }%
%BeginExpansion
\mathbb{R}
%EndExpansion
)}\text{.}
\end{equation*}

In analogy to the \tfs s $\pi _{\lambda }$, we use the Bargmann-Fock shifts $%
\beta _{\lambda }$ defined for functions on $\bC$ by 
\begin{equation*}
\beta _{\lambda }F(z)=e^{\pi i\lambda _{1}\lambda _{2}}e^{\pi \bar{\lambda}%
z}F(z-\lambda )\,e^{-\pi |\lambda |^{2}/2}\,.
\end{equation*}%
We observe that the true polyanalytic Bargmann transform intertwines the \tfs%
\ $\pi _{\lambda }$ and the Bargmann-Fock representation $\beta _{\lambda }$
on $\cF^{n}$ by a calculation similar to \cite[p.~185]{Charly}:%
\begin{equation}
\mathcal{B}^{n}\left( \pi _{\lambda }\gamma \right) (z)=\beta _{\lambda }%
\mathcal{B}^{n}\gamma (z)\,,  \label{intertwining}
\end{equation}%
for $\gamma \in L^{2}(\bR)$. For a countable subset $\Lambda \in 
%TCIMACRO{\U{211d} }%
%BeginExpansion
\mathbb{R}
%EndExpansion
^{2}$, one says that the Gabor system $\mathcal{G}\left( h_{n},\Lambda
\right) =\{\pi _{\lambda }h_{n}:\lambda \in \Lambda \}$ is a \emph{Gabor
frame} or \emph{Weyl-Heisenberg frame }in $L^{2}(%
%TCIMACRO{\U{211d} }%
%BeginExpansion
\mathbb{R}
%EndExpansion
)$,\emph{\ }whenever there exist constants $A,B>0$ such that, for all $f\in
L^{2}(%
%TCIMACRO{\U{211d} }%
%BeginExpansion
\mathbb{R}
%EndExpansion
)$,%
\begin{equation}
A\left\Vert f\right\Vert _{L^{2}(%
%TCIMACRO{\U{211d} }%
%BeginExpansion
\mathbb{R}
%EndExpansion
)}^{2}\leq \sum_{\lambda \in \Lambda }\left\vert \left\langle f,\pi
_{\lambda }h_{n}\right\rangle _{L^{2}(%
%TCIMACRO{\U{211d} }%
%BeginExpansion
\mathbb{R}
%EndExpansion
)}\right\vert ^{2}\leq B\left\Vert f\right\Vert _{L^{2}(%
%TCIMACRO{\U{211d} }%
%BeginExpansion
\mathbb{R}
%EndExpansion
)}^{2}.  \label{frame}
\end{equation}

\subsection{A polyanalytic interpolation formula for $\mathcal{F}^{n}(%
%TCIMACRO{\U{2102} }%
%BeginExpansion
\mathbb{C}
%EndExpansion
)$.}

Consider the lattice $\Lambda =\{m_{1}\lambda _{1}+m_{2}\lambda
_{2};m_{1},m_{2}\in 
%TCIMACRO{\U{2124} }%
%BeginExpansion
\mathbb{Z}
%EndExpansion
\}\subset 
%TCIMACRO{\U{2102} }%
%BeginExpansion
\mathbb{C}
%EndExpansion
$ spanned by the periods $\lambda _{1},\lambda _{2}\in 
%TCIMACRO{\U{2102} }%
%BeginExpansion
\mathbb{C}
%EndExpansion
$, where $\func{Im}(\lambda _{1}/\lambda _{2})>0$. The size of $\Lambda $ is
the area of the parallelogram spanned by $\lambda _1$ and $\lambda _2$ in $%
\bC $. If we identify $\bR ^2$ and $\bC$, then we can write $\Lambda $ as $%
\Lambda = A\bZ ^2$ where $A=[\lambda _{1},\lambda _{2}]$ is an invertible
real $2\times 2$ matrix. Then the size of the lattice is $s(\Lambda )=| \det
A |$.

Let $\sigma $ be the Weierstrass sigma function corresponding to $\Lambda $
defined by 
\begin{equation*}
\sigma (z) = z\prod_{\lambda \in \Lambda \backslash \{0\}}\left( 1-\frac{z}{
\lambda }\right) e^{\frac{z}{\lambda }+\frac{z^{2}}{2\lambda ^{2}}} \, . 
\end{equation*}

% \begin{eqnarray*}
% \sigma (z) &=&z\prod_{\lambda \in \Lambda \backslash \{0\}}\left( 1-\frac{z}{%
% \lambda }\right) e^{\frac{z}{\lambda }+\frac{z^{2}}{2\lambda ^{2}}} \\
% \zeta (z) &=&\frac{1}{z}+\sum_{\lambda \in \Lambda \backslash \{0\}}\left\{
% \frac{1}{z-\lambda }+\frac{1}{\lambda }+\frac{z}{\lambda ^{2}}\right\} \text{%
% .}
% \end{eqnarray*}%
% Defining the constant $\eta _{1}=\zeta (z+\lambda _{1})-\zeta (z)$ and $\eta
% _{2}$ in a similar way, set%
% \begin{equation*}
% a\left( \Lambda \right) =\frac{1}{2}\frac{\eta _{2}\overline{\lambda _{1}}%
% -\eta _{1}\overline{\lambda _{2}}}{\lambda _{1}\overline{\lambda _{2}}%
% -\lambda _{2}\overline{\lambda _{1}}}\text{.}
% \end{equation*}%
% This allows us to define the \emph{modified sigma function} $\sigma
% _{\Lambda }(z)$ associated to the lattice $\Lambda $ as
% \begin{equation*}
% \sigma _{\Lambda }(z)=\sigma (z)e^{a\left( \Lambda \right) z^{2}}\text{.}
% \end{equation*}

% Proposition 3.5 in \cite{CharlyYurasuper} provides a precise growth estimate
% for the modified sigma function. For all $z\in
% %TCIMACRO{\U{2102} }%
% %BeginExpansion
% \mathbb{C}
% %EndExpansion
% $, one has
% \begin{equation}  \label{eq:c3}
% \left\vert \sigma _{\Lambda }(z)\right\vert \lesssim e^{\frac{\pi }{%
% 2s(\Lambda )}\left\vert z\right\vert ^{2}}\text{.}
% \end{equation}

It is then possible to choose an exponent $a = a(\Lambda ) \in \bC $, such
that the modified sigma function $\sigma  _{\Lambda }(z)$ associated to the
lattice $\Lambda $ 
\begin{equation*}
\sigma _{\Lambda }(z)=\sigma (z)e^{a\left( \Lambda \right) z^{2}}\text{}
\end{equation*}
satisfies the growth estimate 
\begin{equation}  \label{eq:c3}
\left\vert \sigma _{\Lambda }(z)\right\vert \lesssim e^{\frac{\pi }{%
2s(\Lambda )}\left\vert z\right\vert ^{2}}\text{.}
\end{equation}
See for instance Proposition 3.5 in \cite{CharlyYurasuper}. We will only
work with the modified sigma function $\sigma _\Lambda $.

Our discussion of sampling theorems for polyanalytic functions will be based
on the following function related to the Weierstrass sigma function: 
% $\sigma _{\Lambda }(z):$
\begin{equation}  \label{intexp}
S_{\Lambda }^n(z)=e^{\pi \left\vert z\right\vert ^{2}}\left( \frac{d}{dz}%
\right) ^{n}\left[ e^{-\pi \left\vert z\right\vert ^{2}}\frac{ \sigma
_{\Lambda }(z) ^{n+1}}{n! \, z}\right]
\end{equation}%
Then by definition $S^n_\Lambda $ is polyanalytic of order $n+1$.

Before stating our result, recall that the set $\Lambda $ is \emph{an
interpolating sequence} for $\mathcal{F}^{n}(%
%TCIMACRO{\U{2102} }%
%BeginExpansion
\mathbb{C}
%EndExpansion
)$ if, for every sequence $\{a_{\lambda }\}_{\lambda \in \Lambda }\in \ell
^{2}(\Lambda )$, there exists $F\in \mathcal{F}^{n}(%
%TCIMACRO{\U{2102} }%
%BeginExpansion
\mathbb{C}
%EndExpansion
)$ such that 
\begin{equation*}
F(\lambda )\, e^{-\frac{\pi }{2}\left\vert \lambda \right\vert ^{2}} =a
_{\lambda },
\end{equation*}%
for every $\lambda \in \Lambda .$\ 

By means of $S^n_\Lambda $ we can now formulate an explicit solution to the
interpolation problem on $\Lambda $ for $\cF ^n$.

\begin{theorem}
\label{th1} If $s(\Lambda )>n+1$, then $\Lambda $ is an interpolating
sequence for $\mathcal{F}^{n+1}(%
%TCIMACRO{\U{2102} }%
%BeginExpansion
\mathbb{C}
%EndExpansion
)$. Moreover, the interpolation problem is solved by 
\begin{equation}
F(z)=\sum_{\lambda \in \Lambda }a_{_{\lambda }}e^{\pi \overline{\lambda }%
z-\pi \left\vert \lambda \right\vert ^{2}/2} \, S^n_{\Lambda }(z-\lambda ) ,
\label{series}
\end{equation}
\end{theorem}

\begin{proof}
The growth estimate~\eqref{eq:c3} implies that 
\begin{equation*}
\Big|\frac{\sigma _\Lambda ^{n+1}(z)}{z} \Big| \lesssim e^{\frac{\pi (n+1)}{%
2 s(\Lambda )} \left\vert z\right\vert ^{2}}\text{.}
\end{equation*}
Since $s(\Lambda )>n+1$, we have $\sigma _{\Lambda }(z) ^{n+1}/z \in 
\mathcal{F}_{2}(%
%TCIMACRO{\U{2102} }%
%BeginExpansion
\mathbb{C}
%EndExpansion
)$. By unitarity of the Bargmann transform, there exists a $\gamma = \gamma
_{n, \Lambda } = \gamma _\Lambda \in L^{2}\left( 
%TCIMACRO{\U{211d} }%
%BeginExpansion
\mathbb{R}
%EndExpansion
\right) $ such that $\mathcal{B} \gamma _{\Lambda }(z) = \sigma _{\Lambda
}(z) ^{n+1}/z$. Furthermore, since $|V_{h_0} \gamma _\Lambda (z)| = |\cB %
\gamma _\Lambda (z)| \, e^{-\pi |z|^2/2}$, it follows that $\gamma _\Lambda
\in M^1(\bR )$ (or even in the Schwartz class).

Comparing~\eqref{eq:hmhm} and \eqref{intexp} we find that 
% Now the definition of the true polyanalytic Bargmann transform yields that
\begin{equation*}
S_{\Lambda }^{n}(z)=\left( \frac{\pi ^{n}}{n!}\right) ^{\frac{1}{2}}(%
\mathcal{B}^{n+1}\gamma _{\Lambda })(z)\text{,}
\end{equation*}%
and $S_{\Lambda }^{n}\in \mathcal{F}_{2}^{n+1}(%
%TCIMACRO{\U{2102} }%
%BeginExpansion
\mathbb{C}
%EndExpansion
)$. As in the proof of \cite[Thm.~1.1]{CharlyYurasuper} we show that $%
S_{\Lambda }^{n}$ is interpolating on $\Lambda $. Using the Leibniz formula,
we expand $S_{\Lambda }^{n}$ as 
\begin{equation*}
S_{\Lambda }^{n}(z)=\sum_{k=0}^{n}\binom{n}{k}(-\pi \bar{z})^{k}\Big(\frac{d%
}{dz}\Big)^{n-k}\left( \frac{\sigma _{\Lambda }^{n+1}(z)}{n!\,z}\right) 
\end{equation*}%
Since $\sigma _{\Lambda }^{n+1}(z)/z$ has zeros of order $n+1$ at $\lambda
\in \Lambda \setminus \{0\}$ and a zero of order $n$ at $\lambda =0$, it
follows that $S_{\Lambda }^{n}(\lambda )=\delta _{\lambda ,0}$ for $\lambda
\in \Lambda $. Consequently, if $F$ is defined by~\eqref{series}, then $%
F(\lambda )e^{-\pi |\lambda |^{2}/2}=a_{\lambda }$, and $F$ is indeed an
interpolation of the sequence $(a_{\lambda })$.

It remains to show that the interpolation series \eqref{series} converges in 
$\cF ^{n+1}$. For this we need the additional information that $\gamma
_\Lambda $ is in $M^1(\bR )$ and the intertwining property~%
\eqref{intertwining}. Since $\gamma _\Lambda \in M^1(\bR )$, the series $%
\sum _{\lambda \in \Lambda } a_\lambda \pi _\lambda \gamma _\Lambda $
converges unconditionally in $L^2(\bR)$ (e.g., by \cite[Thm.~12.2.4]{Charly}%
). Therefore the series 
\begin{eqnarray*}
F(z)&=&\sum_{\lambda \in \Lambda }a_{_{\lambda }}e^{\pi \overline{\lambda }%
z-\pi \left\vert \lambda \right\vert ^{2}/2} \, S^n_{\Lambda }(z-\lambda ) \\
&=& \sum_{\lambda \in \Lambda }a_{_{\lambda }} \beta _\lambda S_\Lambda ^n
(z) \\
&=& \sum_{\lambda \in \Lambda }a_{_{\lambda }} \beta _\lambda \cB ^{n+1}
\gamma _\Lambda (z) \\
&=& \sum_{\lambda \in \Lambda }a_{_{\lambda }} \cB ^{n+1} (\pi _\lambda
\gamma _\Lambda ) (z) \\
&=& \cB ^{n+1} \Big(\sum_{\lambda \in \Lambda }a_{\lambda } \pi _\lambda
\gamma _\Lambda \Big) (z) \,
\end{eqnarray*}
converges in $\cF ^{n+1}$, and the proof is completed.
\end{proof}

\subsection{Gabor frames with Hermite functions on $L^{2}(%
%TCIMACRO{\U{211d} }%
%BeginExpansion
\mathbb{R}
%EndExpansion
)$}

Following Feichtinger and Kozek \cite{FK}, the adjoint lattice $\Lambda ^{0}$
is defined by the commuting property as%
\begin{equation*}
\Lambda ^{0}=\{\mu \in 
%TCIMACRO{\U{211d} }%
%BeginExpansion
\mathbb{R}
%EndExpansion
^{2}:\pi _{\lambda }\pi _{\mu }=\pi _{\mu }\pi _{\lambda }\text{, for all }%
\lambda \in \Lambda \}\text{.}
\end{equation*}%
If $\Lambda =\alpha 
%TCIMACRO{\U{2124} }%
%BeginExpansion
\mathbb{Z}
%EndExpansion
\times \beta 
%TCIMACRO{\U{2124} }%
%BeginExpansion
\mathbb{Z}
%EndExpansion
$, then $\Lambda ^{0}=\beta ^{-1}%
%TCIMACRO{\U{2124} }%
%BeginExpansion
\mathbb{Z}
%EndExpansion
\times \alpha ^{-1}%
%TCIMACRO{\U{2124} }%
%BeginExpansion
\mathbb{Z}
%EndExpansion
$.

There exists a remarkable duality between the Gabor systems with respect to $%
\Lambda ^{0}$ and those with respect to $\Lambda $. This is often referred
to as the \emph{Janssen-Ron-Shen duality principle }~\cite%
{JanssenDuality,RonShen}.

\textbf{Theorem A (}\emph{Duality principle}\textbf{). } \emph{The Gabor
system $\mathcal{G}(g,\Lambda )$ is a \emph{frame} for $L^{2}(%
%TCIMACRO{\U{211d} }%
%BeginExpansion
\mathbb{R}
%EndExpansion
)$ if and only if the Gabor system $\mathcal{G}(g,\Lambda ^{0})$ is a Riesz
basis for its closed linear span in $L^{2}(%
%TCIMACRO{\U{211d} }%
%BeginExpansion
\mathbb{R}
%EndExpansion
)$.}

\vspace{ 3mm}

Combining the duality principle with Theorem 1, one recovers the result of~%
\cite{CharlyYura}:

\begin{theorem}
If $s(\Lambda )<\frac{1}{n+1}$, then the Gabor system $\mathcal{G}%
(h_{n},\Lambda )$ is a frame for $L^{2}(%
%TCIMACRO{\U{211d} }%
%BeginExpansion
\mathbb{R}
%EndExpansion
)$.
\end{theorem}

\begin{proof}
First observe that $s\left( \Lambda ^{0}\right) =\frac{1}{s(\Lambda )}$. If $%
s(\Lambda )<\frac{1}{n+1}$, then $s\left( \Lambda ^{0}\right) >n+1$. It
follows from Theorem 1 that the lattice $\Lambda ^{0}$ is an interpolating
sequence for $\mathcal{F}^{n+1}_2(%
%TCIMACRO{\U{2102} }%
%BeginExpansion
\mathbb{C}
%EndExpansion
)$. Since%
\begin{equation*}
\left\langle f,\pi _{\lambda }h_{n}\right\rangle _{L^{2}(%
%TCIMACRO{\U{211d} }%
%BeginExpansion
\mathbb{R}
%EndExpansion
)}=V_{h_{n}}f(x,\xi )=e^{i\pi x\xi -\frac{\pi }{2}\left\vert z\right\vert
^{2}}\mathcal{B}^{n+1}f(z)\text{,}
\end{equation*}%
then it is clear that $\mathcal{G}(h_{n},\Lambda ^{0})$ is a Riesz basis for
its linear span in $L^{2}(%
%TCIMACRO{\U{211d} }%
%BeginExpansion
\mathbb{R}
%EndExpansion
)$. By the duality principle, the Gabor system $\mathcal{G}(h_{n},\Lambda )$
is a frame for $L^{2}(%
%TCIMACRO{\U{211d} }%
%BeginExpansion
\mathbb{R}
%EndExpansion
)$.
\end{proof}

\section{Banach frames}

In this section we extend the results about Gabor frame expansions in $L^2(%
\bR )$ and the sampling theorem in $\cF ^n_2(\bC )$ to a class of associated
Banach spaces. This extension can be formulated in terms of Banach frames 
\cite{atomGroch} and can be done conveniently with the theory of localized
frames \cite{Localization_one,FoG05}.

\subsection{Banach frames}

The theory of localized frames asserts that every ``nice'' frame is
automatically a Banach frame for an associated class of Banach spaces.

For the description of \modsp s with Gabor frames we recall a precise
statement from \cite[Thm.\ 9]{Localization_one}.

\begin{theorem}
\label{loc} Assume that $\Lambda \subseteq \bR ^2$ is a lattice, that $g\in
M^1(\bR )$, and that $\{\pi _\lambda g : \lambda \in \Lambda \}$ is a frame
for $L^2(\bR ) $.

Then there exists a dual window $\gamma \in M^1(\bR)$, such that the
corresponding frame expansion 
\begin{equation}  \label{eq:1}
f = \sum _{\lambda \in \Lambda } \langle f, \pi _\lambda g \rangle \pi
_\lambda \gamma
\end{equation}
converges unconditionally in $M^p (\bR )$ for $1 \leq p < \infty $ (and weak-%
$^*$ in $M^\infty (\bR )$ ).

A distribution $f$ belongs to the \modsp\ $M^p(\bR )$, \fif\ the frame
coefficients $\langle f, \pi _\lambda g\rangle $ belong to $\ell ^p(\Lambda
) $. Furthermore, the following norm equivalence holds on $M^p(\bR )$: 
\begin{equation}  \label{eq:2}
\|f\|_{M^p} \asymp \Big(\sum _{\lambda \in \Lambda } |\langle f, \pi
_\lambda g\rangle |^p \Big) ^{1/p}
\end{equation}
\end{theorem}

As a consequence of the duality theory (Theorem~A) the dual window $\gamma $
satisfies the biorthogonality condition $s(\Lambda )\inv \, \langle \gamma ,
\pi _\mu g \rangle = \delta _{\mu , 0}$ for $\mu \in \Lambda ^0$.

For $p=2$ the properties \eqref{eq:1} and \eqref{eq:2} are the defining
properties of a frame of a Hilbert space. By analogy for $p\neq 2$, we call
a set satisfying \eqref{eq:1} and \eqref{eq:2} Banach frame for $M^p(\bR )$.

In this paper we have restricted ourselves to dimension $d=1$ and the
unweighted case, but the theory of localized frames offers much more general
versions of Theorem~\ref{loc}.

\subsection{Explicit sampling formulas in $\mathcal{F}_{p}^{n}(%
%TCIMACRO{\U{2102} }%
%BeginExpansion
\mathbb{C}
%EndExpansion
)$}

In this section we translate Theorem~\ref{loc} into the language of
polyanalytic functions and derive a sampling expansion, which in a sense is
the dual of the polyanalytic interpolation formula (\ref{series}) of Theorem
1.

\begin{theorem}
\label{sampfpo} Assume that $\Lambda \subseteq \bC $ is a lattice and $%
s(\Lambda )< (n+1)\inv $.

(i) Then $F$ belongs to the true poly-Fock space $\cF ^n_p (\bC )$, \fif\
the sequence with entries $e^{-\pi |\lambda |^2/2} F(\lambda )$ belongs to $%
\ell ^p (\Lambda )$, with the norm equivalence 
\begin{equation*}
\|F\|_{\cF ^n _p} \asymp \Big(\sum _{\lambda \in \Lambda } |F(\lambda )|^p
e^{-\pi p|\lambda |^2/2} \Big) ^{1/p} \, .
\end{equation*}

(ii) Let 
\begin{equation}
S_{\Lambda ^{0}}^{n}(z)=\left( \frac{\pi ^{n}}{n!}\right) ^{\frac{1}{2}%
}e^{\pi \left\vert z\right\vert ^{2}}\left( \frac{d}{dz}\right) ^{n}\left[
e^{-\pi \left\vert z\right\vert ^{2}}\frac{ \sigma _{\Lambda ^{0}}(z) ^{n+1}%
}{n!z}\right]  \label{eq:c13}
\end{equation}%
be the interpolating function on the adjoint lattice $\Lambda ^{0}$. Then
every $F\in \mathcal{F}_{p}^{n+1}(%
%TCIMACRO{\U{2102} }%
%BeginExpansion
\mathbb{C}
%EndExpansion
)$ can be written as 
\begin{equation}
F(z)=\sum_{\lambda \in \Lambda }F(\lambda )e^{\pi \overline{\lambda }z-\pi
\left\vert \lambda \right\vert ^{2}}S_{\Lambda ^{0}}^{n}(z-\lambda )\,.
\label{eq:c6}
\end{equation}%
The sampling expansion converges in the norm of $\cF_{p}^{n}(\bC)$ for $%
1\leq p<\infty $ and pointwise for $p=\infty $.
\end{theorem}

\begin{proof}
By Corollary~2 there exists an $f\in M^{p}(\bR)$, such that $F=\cB^{n+1}f\in %
\cF_{p}^{n}(\bC)$, more precisely, according to \eqref{polyBargmann} $\cB%
^{n+1}F(z)=e^{-i\pi x\xi }e^{-\pi |z|^{2}/2}\langle f,\pi _{(x,-\xi
)}h_{n}\rangle $. Since $\pi _{\lambda }h_{n}$ is a Banach frame for $M^{p}(%
\bR)$ by Theorem~\ref{loc}, the norm equivalence 
\begin{equation*}
\Vert f\Vert _{M^{p}}\asymp \Big(\sum_{\lambda \in \Lambda }|\langle f,\pi
_{\lambda }h_{n}\rangle |^{p}\Big)^{1/p}
\end{equation*}%
translates into the norm equivalence 
\begin{equation*}
\Vert F\Vert _{\cF_{p}^{n}}\asymp \Big(\sum_{\lambda \in \Lambda }|F(\lambda
)|^{p}e^{-\pi p|z|^{2}/2}\Big)^{1/p},\qquad \qquad \forall F\in \cF_{p}^{n}(%
\bC)\,.
\end{equation*}

We now apply the polyanalytic Bargmann transform to \eqref{eq:1} and obtain
a reconstruction formula for the samples of $F\in \mathcal{F}^{n+1}_p(%
%TCIMACRO{\U{2102} }%
%BeginExpansion
\mathbb{C}
%EndExpansion
)$: 
\begin{equation*}
F(z) = \cB ^{n+1}f(z) =\sum_{\lambda \in \Lambda }F(\lambda )e^{-\frac{\pi }{%
2}\left\vert \lambda \right\vert ^{2}}e^{-\pi i\lambda _{1}\lambda _{2}}%
\mathcal{B}^{n+1}\left( \pi _\lambda\gamma \right) (z).
\end{equation*}%
Now the intertwining property (\ref{intertwining}) gives 
\begin{eqnarray}
F(z)&=&\sum_{\lambda \in \Lambda }F(\lambda )e^{-\frac{\pi }{2}\left\vert
\lambda \right\vert ^{2}}e^{-\pi i\lambda _{1}\lambda _{2}}\beta _{\lambda }%
\mathcal{B}^{n+1} \gamma (z)  \notag \\
&=& \sum_{\lambda \in \Lambda }F(\lambda )e^{-\pi \left\vert \lambda
\right\vert ^{2}}e^{\pi \overline{\lambda }z}\mathcal{B}^{n+1} \gamma
(z-\lambda )\, .  \label{eq:c9}
\end{eqnarray}
Since the frame expansion \eqref{eq:1} converges in $M^p(\bC)$ and since $%
\cB ^{n+1}$ is an isometry from $M^p(\bR )$ onto $\cF ^n_p$, the sampling
expansion \eqref{eq:c9} must converge in $\cF ^n_p(\bC)$.

The expansion~\eqref{eq:c9} holds for every dual window $\gamma \in M^1(\bR )
$. By choosing the particular window, we can derive a more explicit formula
for $\cB ^{n+1}\gamma $. Since every dual window $\gamma $ satisfies the
biorthogonality relation $\delta _{\mu ,0} = s(\Lambda )\inv \,
\langle\gamma , \pi _\mu h_n\rangle = \cB ^{n+1}\gamma (\bar{\mu })\, e^{\pi
i \mu_1 \mu _2 -\pi |\mu |^2/2}$, the true poly Bargmann transform of $%
\gamma $ is an interpolating functions on $\Lambda ^0$. Of all such
functions we may therefore use the interpolating function $S_{\Lambda ^0}^n
= \cB ^{n+1} \gamma $ defined in \eqref{eq:c13} as the expanding function in %
\eqref{eq:c9}.%  is an interpolating
% function on the adjoint lattice $\Lambda ^0$ and that
% \begin{equation*}
% \mathcal{B}^{n+1} \gamma _{n} (z)=(\pi ^{\left\vert
% n\right\vert }n!)^{-\frac{1}{2}}e^{\pi \left\vert z\right\vert ^{2}}\left(
% \frac{d}{dz}\right) ^{n}\left[ e^{-\pi \left\vert z\right\vert ^{2}}\frac{%
% \left( \sigma _{\Lambda ^{0}}(z)\right) ^{n}}{n!z}\right] =G_{\Lambda
% ^{0}}^{n}(z)\text{}
% \end{equation*}%
\end{proof}

The following is a sampling theorem which can be applied to the vector
valued situation studied in \cite{CharlyYurasuper}. Again we emphasize that
this is an explicit formula while the one obtained with the superframe
representation is not, because we do not know the dual vectorial window
explicitly.

\begin{corollary}
If $s(\Lambda )<\frac{1}{n+1}$, then every $F\in \mathbf{F}_{p}^{n+1}(%
%TCIMACRO{\U{2102} }%
%BeginExpansion
\mathbb{C}
%EndExpansion
)$ can be written as:%
\begin{equation*}
F(z)=\sum_{\lambda \in \Lambda }F(\lambda )e^{\pi \overline{\lambda }z-\pi
\left\vert \lambda \right\vert ^{2}}\mathbf{S}_{\Lambda ^{0}}^{n}(z-\lambda
),
\end{equation*}%
where%
\begin{equation*}
\mathbf{S}_{\Lambda ^{0}}^{n}(z)=\sum_{k=0}^{n}S_{\Lambda ^{0}}^{k}(z).
\end{equation*}
\end{corollary}

\begin{proof}
By Corollary 1, $\mathbf{F}_{p}^{n+1}(%
%TCIMACRO{\U{2102} }%
%BeginExpansion
\mathbb{C}
%EndExpansion
)$ can be written as a direct sum of the spaces $\mathcal{F}_{p}^{k}(%
%TCIMACRO{\U{2102} }%
%BeginExpansion
\mathbb{C}
%EndExpansion
)$. Thus, one can write $\mathbf{F}=\sum_{k=0}^{n}F_{k}$, with $F_{k}\in 
\mathcal{F}^{k+1}(%
%TCIMACRO{\U{2102} }%
%BeginExpansion
\mathbb{C}
%EndExpansion
)$, and the result follows from Theorem~\ref{sampfpo}.
\end{proof}

% If $s(\Lambda )<\frac{1}{n+1}$, then every $F\in \mathcal{F}^{n+1}(%
% %TCIMACRO{\U{2102} }%
% %BeginExpansion
% \mathbb{C}
% %EndExpansion
% )$ can be written as:%
% \begin{equation*}
% F(z)=\frac{1}{n!}\sum_{\lambda \in \Lambda }F(\lambda )e^{\pi \overline{%
% \lambda }z-\pi \left\vert \lambda \right\vert ^{2}}S_{\Lambda
% ^{0}}^{n}(z-\lambda ),
% \end{equation*}%
% where%
% \begin{equation*}
% S_{\Lambda ^{0}}^{n}(z)=(\pi ^{\left\vert n\right\vert }n!)^{-\frac{1}{2}%
% }e^{\pi \left\vert z\right\vert ^{2}}\left( \frac{d}{dz}\right) ^{n}\left[
% e^{-\pi \left\vert z\right\vert ^{2}}\frac{\left( \sigma _{\Lambda
% ^{0}}(z)\right) ^{n}}{n!z}\right] .
% \end{equation*}


\begin{thebibliography}{99}
\bibitem{Abreusampling} L. D. Abreu, \emph{Sampling and interpolation in
Bargmann-Fock spaces of polyanalytic functions}, Appl. Comp. Harm. Anal., 29
(2010), 287-302.

\bibitem{Abreustructure} L. D. Abreu, \emph{On the structure of Gabor and
super Gabor spaces}, Monatsh. Math., 161, No. 3, 237-253 (2010).

\bibitem{AAZ} N. Askour, A. Intissar, Z. Mouayn, \emph{Espaces de Bargmann g%
\'{e}n\'{e}ralis\'{e}s et formules explicites pour leurs noyaux
reproduisants.} C. R. Acad. Sci. Paris S\'{e}r. I Math. 325 (1997), no. 7,
707--712. 

\bibitem{Balan} R. Balan, \emph{Multiplexing of signals using superframes},
In SPIE \emph{Wavelets applications, }volume 4119 of Signal and Image
processing XIII, pag. 118-129 (2000).

\bibitem{Balk} M. B. Balk, \emph{Polyanalytic Functions, }Akad. Verlag,
Berlin (1991).

\bibitem{BegehrHille} H. Begehr, G. N. Hile, \emph{A hierarchy of integral
operators.} Rocky Mountain J. Math. 27 (1997), no. 3, 669--706.

\bibitem{Begehr} H. Begehr, \emph{Orthogonal decompositions of the function
space} $L^{2}(\overline{D},%
%TCIMACRO{\U{2102} }%
%BeginExpansion
\mathbb{C}
%EndExpansion
)$. J. Reine Angew. Math. 549 (2002), 191--219.

\bibitem{BS93} S.~Brekke and K.~Seip. \newblock Density theorems for
sampling and interpolation in the {B}argmann-{F}ock space. {III}. \newblock 
\emph{Math. Scand.}, 73(1):112--126, 1993.

\bibitem{DG10} M.~D\"{o}rfer and K.~Gr\"{o}chenig. \newblock Time-frequency
partitions and characterizations of modulation spaces with localization
operators. \newblock {\em Preprint}. Available from ArXive.

\bibitem{FeiModulation} H.~G. Feichtinger. \newblock Modulation spaces on
locally compact abelian groups. \newblock In \emph{Proceedings of
``International Conference on Wavelets and Applications" 2002}, pages
99--140, Chennai, India, 2003. \newblock Updated version of a technical
report, University of Vienna, 1983.

\bibitem{Fei} H. G. Feichtinger, \emph{On a new Segal algebra}. Monatsh.
Math. 92 (1981), no. 4, 269--289. 
% Gabor analysis and algorithms, 233--266,
% Appl. Numer. Harmon. Anal., Birkh\"{a}user Boston, Boston, MA, 1998.

\bibitem{FeiZim} H. G. Feichtinger, G. A Zimmermann, \emph{A Banach space of
test functions for Gabor analysis.} Gabor analysis and algorithms, 123--170,
Appl. Numer. Harmon. Anal., Birkh\"{a}user Boston, Boston, MA, 1998.

\bibitem{FG} H. G. Feichtinger, K. Gr\"{o}chenig, \emph{Banach spaces
related to integrable group representations and their atomic decompositions,
I}, J. Funct. Anal. \textbf{86} (2), 307-340 (1989).

\bibitem{FG0} H. G. Feichtinger, K. Gr\"{o}chenig, \emph{A unified approach
to atomic decompositions via integrable group representations, }Proc.
Function Spaces and Applications, Conf. Lund 1986, Lect.Notes Math. 1302, p.
5273, Springer (1988).

\bibitem{FK} H. G. Feichtinger, W. Kozek, \emph{Quantization of TF
lattice-invariant operators on elementary LCA groups.} Gabor analysis and
algorithms, 233--266, Appl. Numer. Harmon. Anal., Birkh\"{a}user Boston,
Boston, MA, 1998.

\bibitem{folland89} G.~B. Folland. 
\newblock {\em Harmonic Analysis in Phase
Space}. \newblock Princeton Univ. Press, Princeton, NJ, 1989.

\bibitem{Folland} G. B. Folland, \emph{The abstruse meets the applicable:
some aspects of time-frequency analysis.} Proc. Indian Acad. Sci. Math. Sci.
116 (2006), no. 2, 121--136.

\bibitem{FoG05} M.~Fornasier and K.~Gr{\"o}chenig. \newblock Intrinsic
localization of frames. \newblock {\em Constr. Approx.}, 22(3):395--415,
2005.

\bibitem{Fuhr} H. F\"{u}hr, \emph{Simultaneous estimates for vector-valued
Gabor frames of Hermite functions.} Adv. Comput. Math. 29 , no. 4, 357--373,
(2008).

\bibitem{Hunik} O. Hutn\'{\i}k, M. Hutn\'{\i}kov\'{a} \ \ \emph{An
alternative description of Gabor spaces and Gabor-Toeplitz operators} \ Rep.
Math. Phys. 66(2) (2010), 237-250

\bibitem{Luef} F. Luef, \emph{Projective modules over noncommutative tori
are multi-window Gabor frames for modulation spaces.} J. Funct. Anal. 257
(2009), no. 6, 1921--1946.

\bibitem{G} I. Gertner, G. A. Geri, \emph{Image representation using Hermite
functions}, Biological Cybernetics, Vol. 71, 2 , 147-151, (1994).

\bibitem{atomGroch} K. Gr\"{o}chenig, \emph{Describing functions: Atomic
decompositions versus frames, }Monatsh. Math. 112 (1991), 1-42.

\bibitem{Charly} K. Gr\"{o}chenig, \emph{\textquotedblright Foundations of
Time-Frequency Analysis\textquotedblright }, Birkh\"{a}user, Boston, (2001).

\bibitem{GL} K. Gr\"{o}chenig, M. Leinert, \emph{Wiener's lemma for twisted
convolution and Gabor frames.} J. Amer. Math. Soc. 17 (2004), no. 1, 1--18.

\bibitem{Localization_one} K. Gr\"{o}chenig, \emph{\ Localization of frames,
Banach frames, and the invertibility of the frame operator}, J. Fourier
Anal. Appl., 10 (2004), 105--132.

\bibitem{CharlyYura} K. Gr\"{o}chenig, Y. Lyubarskii, \emph{Gabor frames
with Hermite functions}, C. R. Acad. Sci. Paris, Ser. I 344 157-162 (2007).

\bibitem{CharlyYurasuper} K. Gr\"{o}chenig, Y. Lyubarskii, \emph{Gabor
(Super)Frames with Hermite Functions}, Math. Ann. , 345, no. 2, 267-286
(2009).

\bibitem{JanssenDuality} A. J. E. M. Janssen, \emph{Duality and
biorthogonality for Weyl-Heisenberg frames.} J. Fourier Anal. Appl. 1
(1995), no. 4, 403--436.

\bibitem{JPR} S. Janson, J. Peetre, R. Rochberg, \emph{Hankel forms and the
Fock space.} Rev. Mat. Iberoamericana 3 (1987), no. 1, 61--138.

\bibitem{Ramazanov} A. K. Ramazanov, \emph{On the structure of spaces of
polyanalytic functions.} (Russian. Russian summary) Mat. Zametki 72 (2002),
no. 5, 750--764; translation in Math. Notes 72 (2002), no. 5-6, 692--704.

\bibitem{RonShen} A. Ron, Z. Shen, \emph{Weyl-Heisenberg frames and Riesz
bases in }$L^{2}(%
%TCIMACRO{\U{211d} }%
%BeginExpansion
\mathbb{R}
%EndExpansion
^{d})$, Duke Math. J. 89 (1997), 237--282.

\bibitem{SW} K. Seip, R. Wallst\'{e}n, \emph{Density Theorems for sampling
and interpolation in the Bargmann-Fock space II}, J. Reine Angew. Math. 429
(1992), 107-113.

\bibitem{W} J. M. Whittaker, \emph{Interpolatory Function Theory.}
(Cambridge Tracts in Mathematics and Mathematical Physics, No. 33.)
Cambridge University Press. New York, Macmillan, 1935.

\bibitem{VasiFock} N. L. Vasilevski, \emph{Poly-Fock spaces,} Differential
operators and related topics, Vol. I (Odessa, 1997), 371--386, Oper. Theory
Adv. Appl., 117, Birkh\"{a}user, Basel, (2000).
\end{thebibliography}
\end{document}